\documentclass[11pt]{article}
\usepackage{amsfonts}
\usepackage{graphics}
\usepackage{amsmath}
\usepackage{amsthm}

\newtheorem{theorem}{Theorem}[section]
\newtheorem{corollary}{Corollary}[section]
\newtheorem{lemma}{Lemma}[section]

\newtheorem{proposition}{Proposition}[section]

\theoremstyle{remark}

\newtheorem{definition}{Definition}

\begin{document}

\title{The universal semigroup of a $\Gamma$-semigroup}
\author{Elton Pasku\\
elton.pasku@fshn.edu.al}
\date{}

\maketitle

\begin{abstract}
Given a $\Gamma$-semigroup $S$, we construct a semigroup $\Sigma$ in such a way that one sided ideals and quasi-ideals of $S$ can be regarded as one sided ideals and quasi-ideals respectively of $\Sigma$. This correspondence and other properties of $\Sigma$, allow us to obtain several results for $S$ without having the need to work directly with it, but solely employing well known results of semigroup theory. For example, we obtain the Green's theorem for $\Gamma$-semigroups found in \cite{PT}, as a corollary of the usual Green's theorem in semigroups. Also we prove that, if $S$ is a $\Gamma$-semigroup and $\gamma_{0} \in \Gamma$ such that $S_{\gamma_{0}}$ is a completely simple semigroup, then for every $\gamma \in \Gamma$, $S_{\gamma}$ is completely simple too.
\end{abstract}

\section*{Introduction}

Let $S$ and $\Gamma$ be two non empty sets. Every map from $S \times \Gamma \times S$ to $S$ will be called a $\Gamma$-multiplication in $S$ and is denoted by $(\cdot)_{\Gamma}$. The result of this multiplication for $a,b \in S$ and $\gamma \in \Gamma$ is denoted by $a\gamma b$. According to Sen and Saha \cite{GS-I}, a $\Gamma$-semigroup $S$ is an ordered pair $(S,(\cdot)_{\Gamma})$ where $S$ and $\Gamma$ are non empty sets and $(\cdot)_{\Gamma}$ is a $\Gamma$-multiplication on $S$ which satisfies the following property
\begin{align*}
\forall (a,b,c,\alpha, \beta) \in S^{3} \times \Gamma^{2}, (a \alpha b)\beta c= a \alpha (b \beta c).
\end{align*}
Here we give a few notions and present some auxiliary results that will be used throughout the paper. Some of the results may be found in \cite{saha} and \cite{GS-I} but for the reader's convenience we have included their proofs here.

Let $S$ be a $\Gamma$-semigroup and $A,B$ subset of $S$. We define the set
\begin{equation*}
A \Gamma B =\{a \gamma b | a,b \in S \text{ and } \gamma \in \Gamma \}.
\end{equation*}
For simplicity we write $a \Gamma B$ instead of $\{a\} \Gamma B$ and similarly we write $A \Gamma b$, and write $A \gamma B$ in place of $A \{\gamma\} B$.
\begin{definition} \cite{saha}
Let $S$ be a $\Gamma$-semigroup. A non empty subset $S_{1}$ of $S$ is said to be a $\Gamma$-subsemigroup of $S$ if $S_{1} \Gamma S_{1} \subseteq S_{1}$.
\end{definition}
\begin{definition} \cite{GS-I}
A right [left] ideal of a $\Gamma$-semigroup $S$ is a non empty subset $R [L]$ of $S$ such that $R \Gamma S \subseteq R$, $[S \Gamma L \subseteq L]$.
\end{definition}
By analogy with the definition of quasi-dieals in plain semigroups \cite{OSF} we give the following
\begin{definition}
A quasi-ideal of a $\Gamma$-semigroup $S$ is a non empty subset $Q$ of $S$ such that $Q \Gamma S \cap S \Gamma Q \subseteq Q$.
\end{definition}

\section{The semigroup $\Sigma$}

We define $\Sigma$ as the quotient semigroup of the free semigroup $F$ on the set $S \cup \Gamma$ by the congruence generated from the set of ralations
\begin{equation*}
(\gamma_{1},\gamma_{2}) \sim \gamma_{1}, (x, \gamma, y) \sim x \gamma y, (x,y) \sim x \gamma_{0}y
\end{equation*}
for all $\gamma_{1}, \gamma_{2} \text { and } \gamma \in \Gamma$, all $x,y \in S$ and with $\gamma_{0} \in \Gamma$ a fixed element.
\begin{lemma} \label{sigma}
Every element of $\Sigma$ can be represented by an irreducible word which has the form $\gamma x \gamma'$, $\gamma x$, $x \gamma$, $\gamma$ or $x$ where $x \in S$ and $\gamma, \gamma' \in \Gamma$.
\end{lemma}
\begin{proof}
To prove the lemma, we show first that the reduction system arising from the given presentation is Noetherian and confluent, and therefore any element of $\Sigma$ is given by a unique irreducible word from $F$. Secondly, we show that the irreducible words have one of the five claimed forms. \newline
If a word $w$ of $F$ has the form $(u, \gamma_{1},\gamma_{2},v)$ where $\gamma_{1},\gamma_{2} \in \Gamma$, and $u,v$ are possibly empty words of $F$, then $w$ reduces to $w'=(u, \gamma_{1},v)$. Now if for some $x,y \in S$ and $\gamma \in \Gamma$, the word $w$ contains a subword of the form $(x,\gamma, y)$, which is to say that $w=(u, x,\gamma, y, v)$ with $u,v$ being possibly empty words from $F$, then it reduces to $w'=(u, x\gamma y, v)$. Finally, if the word $w$ contains two adjacent letters from $S$, meaning that $w=(u, x, y, v)$ where $u$ and $v$ as before and $x,y \in S$, then it reduces to $w'=(u, x \gamma_{0} y, v)$. In this way we obtain a reduction system made of the following three type of reductions:
\begin{equation*} 
\left.\begin{array}{ccc} (u,  \gamma_{1},\gamma_{2},v) & \rightarrow & (u,\gamma_{1},v) \\
(u, x,\gamma, y, v)  & \rightarrow & (u, x\gamma y, v) \\
(u, x, y, v) & \rightarrow & (u, x \gamma_{0} y, v) \end{array}\right.
\end{equation*}
which is length reducing and therefore Noetherian. To prove that it is confluent, from Newman's lemma, it is sufficient to show that it is locally confluent. 
As there are no inclusion ambiguities, we need to check only overlapping ones. There are only five such pairs: \newline
1- $(x,y,\gamma,z) \rightarrow (x \gamma_{0}y,\gamma, z)$ and $(x,y,\gamma,z) \rightarrow (x , y\gamma z)$. Both resolve to $(x \gamma_{0}y\gamma z)$. \newline
2- $(x, \gamma, y,z) \rightarrow (x, \gamma, y \gamma_{0} z)$ and $(x, \gamma, y,z) \rightarrow (x \gamma y,  z)$ which resolve to $(x \gamma y \gamma_{0} z)$. \newline
3- $(x, \gamma, y, \gamma',z) \rightarrow (x, \gamma, y \gamma' z)$ and $(x, \gamma, y, \gamma',z) \rightarrow (x \gamma y, \gamma', z)$ which resolve to $(x \gamma y \gamma' z)$. \newline
4- $(x,y,z) \rightarrow (x \gamma_{0}y, z)$ and $(x,y,z) \rightarrow (x,y \gamma_{0}z)$, which resolve to $(x \gamma_{0} y \gamma_{0}z)$. \newline
5- $(\gamma_{1},\gamma_{2},\gamma_{3}) \rightarrow (\gamma_{1},\gamma_{3})$ and $(\gamma_{1},\gamma_{2},\gamma_{3}) \rightarrow (\gamma_{1},\gamma_{2})$ which resolve to $(\gamma_{1})$. \newline
To complete the proof, we need to show that the irreducible words representing elements of $\Sigma$ have the forms claimed. This can be achieved easily, by using an inductive argument on the length of the word or by simply observing that those five type of words are the only irreducible words in $F$.
\end{proof}
Lemma \ref{sigma} shows in particular that the natural epimorphism $\mu: F \rightarrow \Sigma$ is injective on $S$ and $\Gamma$. In what follows we will identify the elements of $\Sigma$ with the irreducible words from $F$ they are represented of written without brackets and commas, and if we want to multiply in $\Sigma$ two such words, we take their concatenation and than find its irreducible form. For instance, the product in $\Sigma$ of $x$ with $\gamma y$ is $x \cdot \gamma y=x \gamma y$.

We call $\Sigma$ the \textit{universal semigroup} of the given $\Gamma$-semigroup. The reason for this steams form the following universal property that $\Sigma$ possesses. 
\begin{theorem}
Let $S$ and $S'$ be both $\Gamma$-semigroups. For every homomorphism of $\Gamma$-semigroups $\varphi: S\rightarrow S'$ identical on $\Gamma$, there is a unique homomorphism of semigroups $\phi: \Sigma \rightarrow \Sigma'$ identical on $\Gamma$ and such that $\phi \mu=\mu' \varphi$.
\end{theorem}
\begin{proof}
Let $f: F(S \cup \Gamma) \rightarrow F(S' \cup \Gamma)$ be the homomorphism of free semigroups induced from $\varphi$. We prove that $\varphi$ induces a homomorphism $\phi:\Sigma \rightarrow \Sigma'$. To do this we need to show that every relation that defines $\Sigma$ lies in the kernel of $\nu'f$ where $\nu': F(S' \cup \Gamma) \rightarrow \Sigma'$ is the canonical homomorphism. Indeed, for the first type of relations $(\gamma_{1},\gamma_{2}) \sim \gamma_{1}$ we have
\begin{align*}
\nu'f(\gamma_{1},\gamma_{2})&=\nu'(\varphi(\gamma_{1}),\varphi(\gamma_{2}))\\
&=\varphi(\gamma_{1})\\
&=\gamma_{1}\\
&=\nu'f(\gamma_{1}).
\end{align*}
For the second type $(x, \gamma, y) \sim x \gamma y$ we have
\begin{align*}
\nu'f(x,\gamma,y)&=\nu'(\varphi(x),\gamma,\varphi(y))\\
&=\varphi(x)\gamma \varphi(y)\\
&=\varphi(x \gamma y)\\
&=\nu'f(x\gamma y),
\end{align*}
and for the last type $(x,y) \sim x \gamma_{0}y$ we have
\begin{align*}
\nu'f(x,y)&=\nu'(\varphi(x),\varphi(y))\\
&=\varphi(x)\gamma_{0} \varphi(y)\\
&=\varphi(x \gamma_{0} y)\\
&=\nu'f(x\gamma_{0} y).
\end{align*}
The equality $\phi \mu=\mu' \varphi$ is obvious and the uniqueness of $\phi$ with the given property follows easily.
\end{proof}
The next two propositions and the subsequent corollary establish a 1-1 correspondence between one sided ideals and quasi ideals of $S$, and their counterparts of $\Sigma$. This correspondence will be useful in the proof of theorem \ref{GT}.
\begin{proposition}
For every $x \in S$, the left ideal in $\Sigma$ generated by $x$ is the set $(x)_{\ell}^{\Sigma}=(x)_{\ell}^{\Gamma} \cup \Gamma (x)_{\ell}^{\Gamma}$ where $(x)_{\ell}^{\Gamma}=S \Gamma x \cup \{x\}$ is the left ideal in $S$ generated by $x$ and $\Gamma (x)_{\ell}^{\Gamma}$ is a short notation for the set $\{ \gamma y: \gamma \in \Gamma \text{ and } y \in (x)_{\ell}^{\Gamma} \}$.
\end{proposition}
\begin{proof}
The elements of $(x)_{\ell}^{\Sigma} \setminus \{x\}$ are of the following five forms: \newline
1- $\gamma y \cdot x$ with $\gamma \in \Gamma$ and $y \in S$. But $\gamma y \cdot x=\gamma y \gamma_{0} x$ which belongs to $\Gamma (x)_{\ell}^{\Gamma}$. \newline
2- $\gamma y \gamma' \cdot x$ with $\gamma, \gamma' \in \Gamma$ and $y \in S$. Again $\gamma y \gamma' \cdot x=\gamma (y \gamma'  x)$ which belongs to $\Gamma (x)_{\ell}^{\Gamma}$. \newline
3- $\gamma \cdot x$ with $\gamma \in \Gamma$ which obviously belongs to $\Gamma (x)_{\ell}^{\Gamma}$. \newline
4- $y \cdot x$ which equals to $y \gamma_{0} x$ and belongs to $(x)_{\ell}^{\Gamma}$. \newline
5- $y \gamma \cdot x$ which equals to $y \gamma x$ and as before belongs to $(x)_{\ell}^{\Gamma}$.
\end{proof}
The analogue of the above for right ideals also holds true.
\begin{proposition}
For every $x \in S$, the right ideal in $\Sigma$ generated by $x$ is the set $(x)_{r}^{\Sigma}=(x)_{r}^{\Gamma} \cup (x)_{r}^{\Gamma} \Gamma$ where $(x)_{r}^{\Gamma}=x \Gamma S \cup \{x\}$ is the right ideal in $S$ generated by $x$ and $(x)_{r}^{\Gamma} \Gamma$ is the short notation for the set $\{ y \gamma : \gamma \in \Gamma \text{ and } y \in (x)_{r}^{\Gamma} \}$.
\end{proposition}
\begin{corollary}
For every $x \in S$, the quasi ideal in $\Sigma$ generated by $x$ is the set $(x)_{q}^{\Sigma}=(x)_{\ell}^{\Gamma} \cap (x)_{r}^{\Gamma}=(x)_{q}^{\Gamma}$.
\end{corollary}
Lemma 1.3 of \cite{PT} (which can be derived from our own approach as well) and the above corollary imply that for every $x \in S$, $\mathcal{H}_{x}^{\Sigma}=\mathcal{H}_{x}^{\Gamma}$. This remark together with the Green's theorem in plain semigroups allow us to find a shorter proof of Theorem 2.1 of \cite{PT} which is the $\Gamma$-semigroup version of the Green theorem. 
\begin{theorem} \label{GT} (Green Theorem)
Suppose that $x,y$ and $x \gamma_{0} y$ for a certain $\gamma_{0} \in \Gamma$ belong to the same class $\mathcal{H}^{\Gamma}_{x}$. Then, $\mathcal{H}^{\Gamma}_{x}$ is a subgroup of the semigroup $S_{\gamma_{0}}$.
\end{theorem}
\begin{proof}
For the particular $\gamma_{0}$ stated in the theorem, we construct the semigroup $\Sigma$ for which we know that $\mathcal{H}^{\Gamma}_{x}$ and $\mathcal{H}_{x}^{\Sigma}$ coincide. Now since $x,y, x \gamma_{0} y \in \mathcal{H}^{\Gamma}_{x}$, we have that $x,y, x \gamma_{0} y \in \mathcal{H}^{\Sigma}_{x}$. But $x \gamma_{0} y=xy$ in $\Sigma$, hence $\mathcal{H}^{\Sigma}_{x}$ satisfies the Green condition and then the Green's theorem for plain semigroups implies that $\mathcal{H}^{\Sigma}_{x}$ is a group. It is now obvious that $\mathcal{H}^{\Gamma}_{x}$ is a subgroup of $S_{\gamma_{0}}$.
\end{proof}
Next we will prove by using $\Sigma$ a generalization of the well known result of \cite{GS-I} which states that, for a given $\Gamma$ semigroup $S$, if $S_{\gamma_{0}}$ is a group for some $\gamma_{0} \in \Gamma$, then $S_{\gamma}$ is a group for every $\gamma \in \Gamma$.

\begin{theorem}
For a given $\Gamma$ semigroup $S$, if  for some $\gamma_{0} \in \Gamma$ $S_{\gamma_{0}}$ is a completely simple semigroup (without zero), then $S_{\gamma}$ is a completely simple semigroup (without zero) for every $\gamma \in \Gamma$.
\end{theorem}
\begin{proof}
As in the previous theorem, we let $\Sigma$ be the universal semigroup constructed for $\gamma_{0}$. We proceed by first showing that $\Sigma'=\Sigma \setminus \Gamma$ is a completely simple semigroup without zero. To show that it is simple, we note that $\Sigma'$ is a disjoint union of subsemigroups of the form $S$, $\gamma S$, $S \gamma$, $\gamma S \gamma'$ for $\gamma, \gamma'$ varying in $\Gamma$. If $J$ is an ideal of $\Sigma'$ containing an $x \in S$, since $S_{\gamma_{0}}$ is a simple semigroup, then it follows easily that $S \subseteq J$, and that $\Sigma'=J$. If $J$ contains some $x \gamma$, then, for any $y \in S$, it contains $x \gamma y$ which lies in $S$, and than we proceed as before. The same argument applies if $J$ contains some $\gamma x$ or some $\gamma x \gamma'$. Now we show that $\Sigma'$ contains a primitive idempotent. Let $\varepsilon_{0}$ be a primitive idempotent of $S_{\gamma_{0}}$. We show that $\gamma_{0} \varepsilon_{0} \gamma_{0}$ is a primitive idempotent of $\Sigma'$. To do this, one should observe that any idempotent which is lower than $\gamma_{0} \varepsilon_{0} \gamma_{0}$ in the natural order must have the form $\gamma_{0} \varepsilon \gamma_{0}$ where $\varepsilon$ is an idempotent in $S_{\gamma_{0}}$. Since $\varepsilon_{0}$ is primitive, it follows that $\varepsilon=\varepsilon_{0}$, and therefore $\gamma_{0} \varepsilon \gamma_{0}=\gamma_{0} \varepsilon_{0} \gamma_{0}$ proving that the latter is primitive. The assertion that $\Sigma'$ does not have a zero is proved easily. Being completely simple, $\Sigma'$ is completely regular, therefore any element of $\Sigma'$ lies in a subgroup of $\Sigma'$. Let $\gamma \in \Gamma$ be ordinary and $\gamma x \in \gamma S$. There is a subgroup $G$ of $\Sigma'$ such that $\gamma x \in \gamma S \cap G$. It follows that the unit of $G$ must have the form $\gamma \varepsilon$ where $\varepsilon$ is an idempotent in $S_{\gamma}$. Taking this into account, one can easily show that any $g \in G$ must have the form $\gamma z$, therefore $G \subseteq \gamma S$. This shows that $\gamma S$ is completely regular. It is easy to show that $S_{\gamma}$ is completely regular too. We show that $S_{\gamma}$ does not contain a zero element. If it does contain a right zero $z$, then, with the convention that $\varepsilon_{z}^{\gamma}$ denotes the unit of the subgroup of $S_{\gamma}$ containing $z$, we have that for every $x \in S_{\gamma_{0}}$,
\begin{equation*}
x \gamma_{0} z=(x \gamma_{0} \varepsilon_{z}^{\gamma}) \gamma z=z,
\end{equation*}
which shows that $z$ is a right zero of $S_{\gamma_{0}}$. In the same way one proves that $S_{\gamma_{0}}$ has a left zero reaching a contradiction.
Next we prove that $S_{\gamma}$ contains a minimal left ideal. Let $L$ be a minimal left ideal of $S_{\gamma_{0}}$. We prove that $L$ is also a minimal left ideal of $S_{\gamma}$. For any $a \in L$ we denote by $\varepsilon_{a}^{\gamma_{0}}$ the unit of the subgroup of $S_{\gamma_{0}}$ containing $a$. For any $x \in S_{\gamma}$ we have that
\begin{equation*}
x \gamma a= x \gamma (\varepsilon_{a}^{\gamma_{0}} \gamma_{0} a)=(x \gamma \varepsilon_{a}^{\gamma_{0}}) \gamma_{0} a
\end{equation*}
which belongs to $L$ since $L$ is left ideal of $S_{\gamma_{0}}$. To show the minimality of $L$ we must show that it is generated by any of its elements. Let $a \in L$ be any element and $(a)_{\ell}^{S_{\gamma}}=\{x \gamma a: x \in S \}\cup \{a\}$ be the left ideal in $S_{\gamma}$ generated by $a$. Since for any $z \in S$, 
\begin{equation*}
z \gamma_{0}(x \gamma a)=(z \gamma_{0} x) \gamma a \in (a)_{\ell}^{S_{\gamma}} \text{ and } z \gamma_{0} a=(z \gamma_{0} \varepsilon_{a}^{\gamma}) \gamma a \in (a)_{\ell}^{S_{\gamma}}
\end{equation*}
then $(a)_{\ell}^{S_{\gamma}}$ is a left ideal of $S_{\gamma_{0}}$ contained in $L$. The simplicity of $L$ implies $L=(a)_{\ell}^{S_{\gamma}}$ proving that $L$ is simple in $S_{\gamma}$ too. Since each $\mathcal{L}$-class is a union of groups in $S_{\gamma}$, then any left ideal $L$ contains an idempotent, and as a result Theorem 5.8 of \cite{OSF} implies that $S_{\gamma}$ contains a minimal quasi-ideal. Theorem 5.12 \cite{OSF} implies that the union of all minimal quasi-ideals of $S_{\gamma}$ is the kernel and is a simple subsemigroup. But Theorem 5.14 \cite{OSF} shows that minimal quasi-ideals are precisely the maximal subgroups of $S_{\gamma}$, therefore the kernel coincides with $S_{\gamma}$ proving that $S_{\gamma}$ is simple. It follows that $S_{\gamma}$ is completely simple.
\end{proof}
We may now define a completely simple $\Gamma$ semigroup as a $\Gamma$ semigroup $S$ having the property that there exists $\gamma_{0}\in \Gamma$ such that $S_{\gamma_{0}}$ is a completely simple semigroup.


\begin{thebibliography}{99}

\bibitem{clif} Clifford, A.H., \textit{Remarks on 0-minimal quasi-ideals in semigroups}, Semigroup Forum 1978, 16, 183-196
\bibitem{dutta} Dutta, T. K. and Charterjee, T. K., \textit{Green's equivalences on $\Gamma$-semigroups}, Bull. Cal. Math. Soc. 80 (1987), 30-35
\bibitem{Newman}  Newman, M. H. A., \textit{On theories with a combinatorial definition of 'equivalence'}, Ann. of Math. 43, No. 2 (1942) 223-243
\bibitem{PT} Petro. P., Xhillari. Th., \textit{Green¡Çs Theorem and Minimal Quasi-Ideals in $\Gamma$-Semigroups}, Int. J. Algebra Vol. 5, No. 10, 461-470, 2011
\bibitem{saha} Saha N. K., \textit{On $\Gamma$-semigroups}, Bull. Cal. Math. Soc. 79 (1987), 331-335
\bibitem{GS-I} Sen M. K., Saha N. K., \textit{On $\Gamma$-semigroups I}, Bull. Cal. Math. Soc. 78 (1986), 180-186
\bibitem{OSF} Steinfeld O., \textit{Quasi-ideals in semigroups and rings}, Akademiai Kiado, Budapest 1978

\end{thebibliography}
\end{document}